\documentclass[10pt, reqno]{amsart}
\def\FillInBlank{\rule{2truein}{.01truein}}
\def\Blank{\rule{1.5truein}{.01truein}}
\usepackage{amsmath,amssymb,amsfonts, amsthm,enumerate} 
\usepackage{verbatim}
\usepackage{graphicx, graphics}
\usepackage{algorithm}
\usepackage{IEEEtrantools}
\usepackage{longtable}
\usepackage{amscd}

\usepackage[all,arc,curve,frame,color]{xy}
\usepackage{subfigure}
\usepackage{url}

\usepackage{tikz}
\usetikzlibrary{arrows}
\tikzstyle{block}=[draw opacity=0.7,line width=1.4cm]
\def\da{\draw[->]}

\newtheorem{thm}{Theorem}[section]
\newtheorem{lem}[thm]{Lemma}

\newtheorem{prop}[thm]{Proposition}

\newtheorem*{conjecture*}{Conjecture}
\newtheorem*{thm*}{Theorem}

\theoremstyle{remark}
\newtheorem*{question*}{Question}

\newtheorem*{remark*}{Remark}

\theoremstyle{definition}
\newtheorem{define}[thm]{Definition}

\usepackage{xcolor}

\newcommand{\G}{\mathcal{G}}

\newcommand{\LL}{\mathcal{L}}     
\newcommand{\MM}{\mathcal{M}}  
\newcommand{\OO}{\mathcal{O}}    
\newcommand{\FF}{\mathbb{F}}      
\newcommand{\ZZ}{\mathbb{Z}}     
\newcommand{\RR}{\mathbb{R}}     
\newcommand{\PP}{\mathbb{P}}      
\newcommand{\Aff}{\mathbb{A}}      
\newcommand{\XX}{\mathcal{X}}      
\newcommand{\QQ}{\mathbb{Q}}      
\newcommand{\CC}{\mathbb{C}}      
\newcommand{\mm}{\mathfrak{m}}   
\newcommand{\pp}{\mathfrak{p}}   
\newcommand{\qq}{\mathfrak{q}}  
\newcommand{\Gm}{\mathbb{G}_m}  
\newcommand{\Ga}{\mathbb{G}_a}  
\newcommand{\hh}{\mathfrak{h}}  
\newcommand{\NN}{\mathbb{N}}  

\newcommand{\bianca}[1]{{\color{magenta} \sf $\clubsuit\clubsuit\clubsuit$ Bianca: [#1]}}
\newcommand{\michelle}[1]{{\color{orange} \sf $\clubsuit\clubsuit\clubsuit$ Michelle: [#1]}}

\newcommand{\End}{\operatorname{End}}
\newcommand{\Hom}{\operatorname{Hom}}
\newcommand{\im}{\operatorname{im}} 
\newcommand{\coker}{\operatorname{coker}}  
\newcommand{\Sym}{\operatorname{Sym}}      
\newcommand{\Spec}{\operatorname{Spec}}
\newcommand{\ord}{\operatorname{ord}}
\newcommand{\Div}{\operatorname{div}}    
\newcommand{\Gal}{\operatorname{Gal}}  
\newcommand{\Gauss}{\operatorname{Gauss}}  
\newcommand{\supp}{\operatorname{supp}}   
\newcommand{\Pic}{\operatorname{Pic}}        
\newcommand{\Jac}{\operatorname{Jac}}       
\newcommand{\mult}{\operatorname{mult}}  
\newcommand{\pr}{\operatorname{pr}}     
\newcommand{\sep}[1]{{#1}^{\operatorname{s}}}    
\newcommand{\Spf}{\operatorname{Spf}}    
\newcommand{\Frac}{\operatorname{Frac}}    
\newcommand{\chern}[1]{c_1\left(#1\right)}   
\newcommand{\codim}{\operatorname{codim}}  
\newcommand{\dist}{\operatorname{dist}}   
\newcommand{\an}[1]{\operatorname{an}}  
\newcommand{\Aut}{\operatorname{Aut}}   
\newcommand{\Rat}{\operatorname{Rat}}    
\newcommand{\PGL}{\operatorname{PGL}}
\newcommand{\PSL}{\operatorname{PSL}}
\newcommand{\alg}[1]{{\overline{#1}}}
\newcommand{\GG}{\mathbb{G}}

\newcommand{\leftexp}[2]{{\vphantom{#2}}^{#1}{#2}}   
\newcommand{\simarrow}{\stackrel{\sim}{\rightarrow}}    
\newcommand{\ip}[2]{\left\langle #1,#2 \right\rangle} 
\newcommand{\into}{\hookrightarrow}     
\newcommand{\dint}{\int \!\!\! \int}   
\newcommand{\tth}{^{\operatorname{th}}}
\newcommand{\Berk}{\mathbf{P}}  

\newcommand{\Manoa}{M\=anoa}
\newcommand{\Hawaii}{Hawai\kern.05em`\kern.05em\relax i}

\newcommand{\id}{\mathrm{id}}
\newcommand{\oo}{\mathfrak{o}}
\DeclareMathOperator{\Per}{Per}
\DeclareMathOperator{\PrePer}{PrePer}
\DeclareMathOperator{\Twist}{Twist}
\DeclareMathOperator{\Ker}{Ker}
\DeclareMathOperator{\lcm}{lcm}

\author{Michelle Manes}
\author{Bianca Thompson}

\address{Michelle Manes: Department of Mathematics, University of Hawaii, 2565 McCarthy Mall, Honolulu, HI 96822, USA}\email{mmanes@math.hawaii.edu}

\address{Bianca Thompson: Department of Mathematics, University of Hawaii, 2565 McCarthy Mall, Honolulu, HI 96822, USA}\email{bat7@hawaii.edu}
 
\thanks{The work of both authors was partially supported by NSF-DMS 1102858.}



\title[Periodic points in towers of finite fields]{Periodic points in towers of finite fields for polynomials associated to algebraic groups}
\date{\today}

\begin{document}
\begin{abstract}
We find the limiting proportion of periodic points in towers of finite fields for polynomial maps associated to algebraic groups, namely pure power maps $\phi(z) = z^d$  and  Chebyshev polynomials.
\end{abstract}

\maketitle

\section{Introduction}
We fix the following notation:

\begin{tabular}{l l }
$\phi(z)$  & a polynomial.\\
$\phi^n(z)$ & the $n\tth$ iterate of $\phi$ under composition; we take $\phi^0(z) = z$.\\
$\OO_\phi(\alpha)$ & the (forward) orbit of a point $\alpha$ under $\phi$; i.e. $ \{ \phi^n(z) \mid n \geq 0 \}$.\\
$\Per(\phi, K)$ & the set of periodic points for $\phi$ in the field $K$;\\
& i.e. $\{ \alpha \in K \mid \phi^n(\alpha) = \alpha \text{ for some } n >0 \}$.\\
\end{tabular}

When iterating a polynomial function $\phi$ over a finite field, the orbit of any point $\alpha \in \FF_{p^n}$
is a finite set.  That is,  all points are preperiodic, meaning the orbit eventually enters a cycle.  But  many natural questions about the structure of orbits over finite fields remain:\label{3questions}
\begin{enumerate}
\item
Fix a finite field $\FF_{p^n}$ and look over all polynomials of fixed degree $d$:  On average are there ``lots'' of periodic points with relatively small tails leading into the cycles?  Or do we expect few periodic points with long tails? (See Figures~\ref{fig:lots of perper} and~\ref{fig:lots of per}.)
\item
Fix a polynomial:  How does the proportion of periodic points in $\FF_p$ vary as $p \to \infty$?
\item
Again fix a polynomial: How does the proportion of periodic points in $\FF_{p^n}$ vary as $n \to \infty$?
\end{enumerate}

\begin{figure}
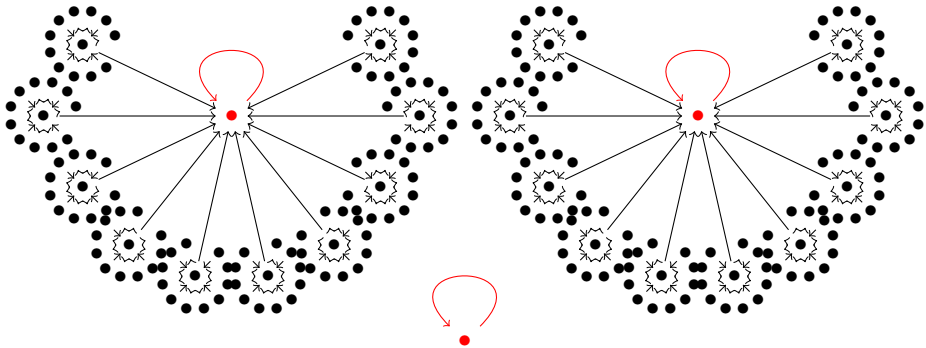


\caption{Few periodic points: $\phi(z) = z^{11}$ on $\FF_{3^5}$ has  three fixed points and 240 strictly preperiodic points.}
\label{fig:lots of perper}
\end{figure}

\begin{figure}
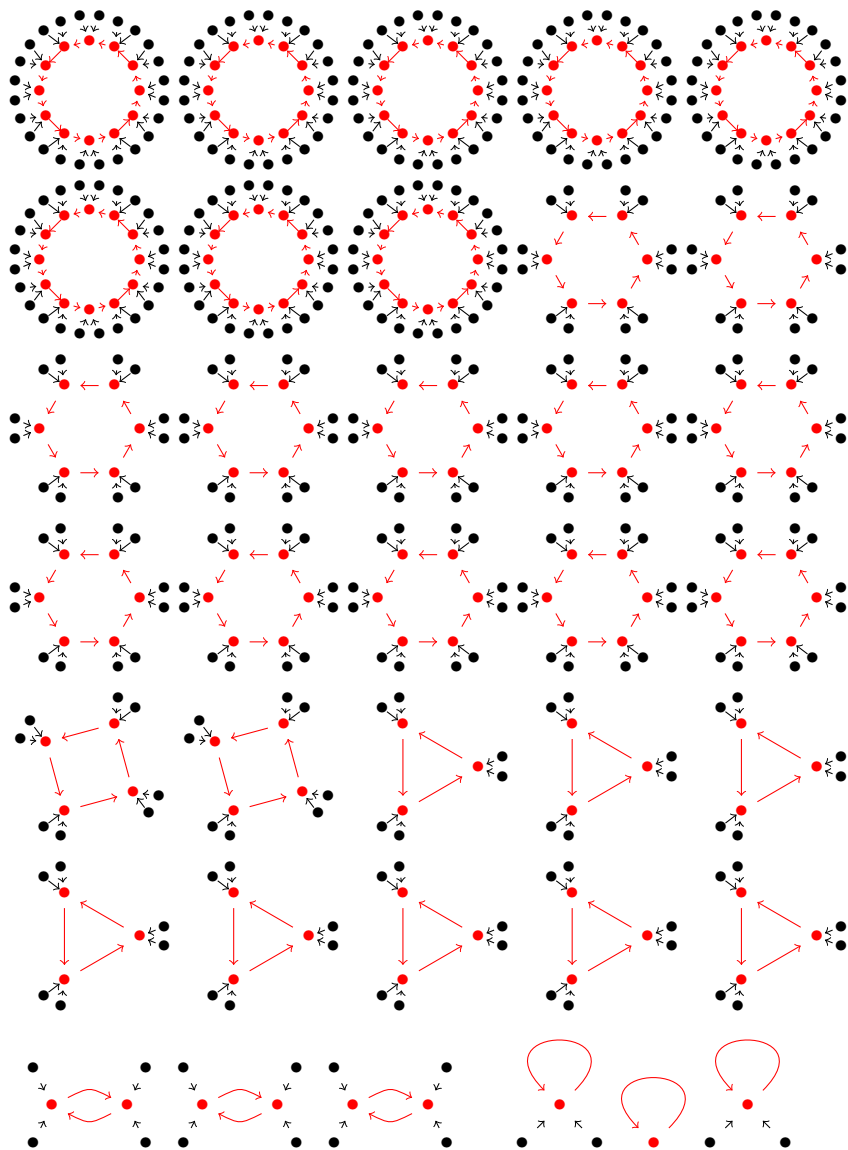

%
\caption{Lots of periodic points: $\phi(z) = z^{3}$ on $\FF_{5^4}$ has 209 periodic points and 416 strictly preperiodic points.}
\label{fig:lots of per}
\end{figure}

Recent work by Flynn and Garton~\cite{flynngarton} addresses the first question.  Using combinatorial arguments, they bound   the average number of periodic points over all polynomials of degree~$d$.  For $d$ large (that is, $d \geq \sqrt{p^n}$), their bound of $\frac 5 6 \sqrt{p^n}$ agrees with earlier heuristic arguments.

In her thesis~\cite{madhu}, Madhu tackles the second question in the case $\phi(z) = z^m + c$, using Galois-theoretic methods.  With some restrictions on $c$, she shows that for primes congruent to~$1$ modulo $m$, the proportion of points in $\FF_p$ that are
periodic points for $\phi$ goes to zero as $p \to \infty$.

In the current work, we focus on the third question in the special case that the polynomial map $\phi(z)$ 
can be viewed as an endomorphism of an underlying algebraic group. This restriction makes the structure of the periodic points particularly simple and is therefore a natural place to begin a more complete investigation of the question.   

We will quickly see that in fact the na\"ive  limit
\[
\lim_{n\to \infty} 
\frac{\#\Per\left(\phi,\FF_{p^n}\right)}{p^n}
\]
does not exist in general, because the map $\phi$ acts as a permutation polynomial whenever $n$ is relatively prime to the multiplicative order of $p$ modulo the degree of $\phi$.  

However, we are able to find limiting proportions along  towers of finite fields $\FF_{p^n}$ with suitable divisibility conditions on $n$.
For example, we have the following two results for $q$ an odd prime.  Similar results hold in the case $q = 2$ and for maps of composite degree.

\begin{thm*}[Theorems~\ref{oddcase} and~\ref{lem: P_v chebyshev}]
Fix a prime $p$ and let  $q$ be a different odd prime.  Define $\delta = \ord_q(p)$ and  $\mu = v_q(p^\delta - 1) \geq 1$.  Let  $\phi(z) = z^q$,  and let $T_q(z)$ be the $q^{\text{th}}$ Chebyshev polynomial.  Then we have the following:
\begin{align*}
 \lim_{\substack{
n \to \infty\\
\delta \mid n\\
v_q(n)=\nu }}\frac{\# \Per\left(\phi, \FF_{p^n}\right)}{p^n}
& = \frac 1{q^{\mu+\nu}}, \text{ and}\\
 \lim_{\substack{
n \to \infty\\
\delta \mid 2n\\
v_q(n)=\nu }}\frac{\# \Per\left(T_q, \FF_{p^n}\right)}{p^n} 
& = \frac{q^{\mu+\nu}+1}{2 q^{\mu+\nu}}.
\end{align*}

\end{thm*}

\subsection*{Outline} 
Section~\ref{sec: alg groups} offers a brief overview of the two families of polynomials we consider here: pure power maps and Chebyshev polynomials.  In Section~\ref{sec: bg} we prove some useful lemmas concerning $q$-adic valuations.   Sections~\ref{sec: power} and~\ref{sec: Cheby} give our main results for pure power maps and Chebyshev polynomials respectively.

\subsection*{Acknowledgements} 
The three questions on page~\pageref{3questions} grew out of Joe Silverman's lectures and problems at the 2010 Arizona Winter School, which the authors were both lucky to attend.
The authors are very grateful for the opportunity to attend Sage Days 42, where this work was begun in earnest.  Thanks especially to our working group for helpful conversations and computations: Alina Bucur, Anna Haensch, Adriana Salerno, Lola Thompson, and Stephanie Treneer.  Thanks also to Tom Tucker and Kalyani Madhu for help with pictures.

\section{Polynomials associated to endomorphisms of algebraic groups}\label{sec: alg groups}

We first consider the multiplicative group $\Gm$ where for a field $K$,  the $K$-valued points are $\Gm(K) = K^*$.  The endomorphism ring of $\Gm$ is $\ZZ$:
\begin{align*}
\ZZ &\to \End(\Gm)\\
d &\mapsto z^d.
\end{align*}
So these pure power maps can be viewed as endomorphisms of an underlying group.  Iteration of pure power maps is particularly easy to understand, as
\[
\phi(z) = z^d \quad \text{ means } \phi^n(z) = z^{d^n}.
\]

Similarly, we consider the additive group $\Ga$, whose underlying scheme is the affine line $\Aff^1$, which may be viewed as a quotient of $\Gm$:
\begin{align*}
\Gm/\{z = z^{-1}\} &\to \Aff^1\\
z &\mapsto z+z^{-1}.
\end{align*}
Since the automorphism $z\mapsto z^{-1}$ commutes with the power map $\phi(z) = z^d$, the polynomial $\phi$ descends to an endomorphism of $\Aff^1$, which we denote $T_d$, the $d\tth$ Chebyshev polynomial.  


\def\zd{{\rm z\mapsto z^d}}\def\Td{{\rm \omega \mapsto T_d(\omega)}} \def\zinverse{{\rm z \mapsto z+z^{-1}}}
\begin{equation*}
\begin{CD}
\Gm @>\zd>> \Gm \\
@VVV @VVV\\
\Gm /z\sim z^{-1} @>\zd>> \Gm /z\sim z^{-1}\\
@VV\zinverse V @VV\zinverse V\\
\Aff^1 @>\Td>> \Aff^1 \\
\end{CD}
\end{equation*}

Taking as a definition the fact that $T_d(w) \in \ZZ[w]$ satisfies
\begin{equation}
T_d(z+z^{-1}) = z^d + z^{-d},
\label{eqn:chebyshev def}
\end{equation}
one may prove existence and uniqueness of the Chebyshev polynomials along with a simple recursion
\begin{equation}
T_d(w) = 
\begin{cases}
2 & d = 0\\
w & d = 1\\
w T_{d-1}(w) - T_{d-2}(w) & d \geq 2.
\end{cases}\label{eqn: Chebyshev}
\end{equation}
A pleasant rule for composition of Chebyshev polynomials arises directly from the definition in~\eqref{eqn:chebyshev def}:
\[
T_d \circ T_e(w) = T_{de}(w) = T_e \circ T_d(w),
\] 
which in turn gives a simple form of iteration
\begin{equation}\label{eqn: cheby iter}
T_d^n(w) = T_{d^n}(w).
\end{equation}

We refer the interested reader to~\cite[Chapter 6]{ads} for more on the dynamics of pure power maps, Chebyshev polynomials, and other rational maps arising from algebraic groups, including proofs of some of the statements above.

\section{Preliminaries}\label{sec: bg}
This section contains a few  facts about valuations and periodic points over finite fields which will be useful in the sequel.  Throughout this section, $p$ and $q$ represent distinct primes, $n$ is a positive integer, and we use the  following additional notation:

\begin{tabular}{l l}
$v_q(n)$ & $q$-adic valuation; i.e. if $n = q^\nu d$ with $q \nmid d$, then $v_q(n) = \nu$.\\
$\delta$ & $\ord_q(p)$; i.e. the smallest positive integer such that $q \mid \left(p^\delta - 1\right)$.\\
\end{tabular}

Since our goal is ultimately to classify periodic points in finite fields, we need to be able to recognize which points are periodic as opposed to strictly preperiodic.  Our first result says that any finite set that is \emph{forward invariant under $\phi$} contains only periodic points.

\begin{lem}\label{lem: finitefieldperpts}
Let $\phi(z) \in K[z]$ be a polynomial and let $S \subseteq K$ be finite.  If 
\[
\phi(S) = S,
\]
 then $S \subseteq \Per(\phi,K).$
\end{lem}

\begin{proof}
Fix $\alpha \in S$.  For every $n > 0$ we have $\phi^n(S) = S$.  Hence for every $n$, we can find $\beta_n \in S$ such that $\phi^n(\beta_n) = \alpha$.

Since  $S$ is  finite, for some $n > m >0$, we must have $\beta_n = \beta_m$.  But this means we have $\beta \in S$ such that
\[
\phi^m(\beta) = \alpha \text{ and }
\phi^n(\beta) = \alpha, \text{ so} \quad
\phi^{n-m}(\alpha)= \alpha
\]
and $\alpha$ is periodic.
\end{proof}

The next three lemmas give us the tools to calculate the $q$-adic valuation of $p^{nd}-1$ based on the valuations of $p^d-1$ and $n$.  These will be used to create the towers of finite fields for which we can calculate limiting proportions of periodic points.  The results are different enough for $q=2$ compared to odd primes that we break up the cases along those lines.

\begin{lem}\label{lem: q notdiv n}
Let $p$ and $q$ be distinct primes.  Suppose $v_q( p^d - 1 ) = \mu \geq 1$ and $v_q(n) = 0$.  Then $v_q( p^{nd} - 1 ) = \mu$.
\end{lem}

\begin{proof}
\begin{align*}
v_q(p^{nd} - 1) & = v_q(p^d-1) + v_q(\underbrace{p^{(n-1)d} + p^{(n-2)d} + \cdots + p^d + 1}_{n \text{ terms, all 1 $\textup{mod } q$. } })\\
& = \mu + 0 = \mu. \qedhere
\end{align*}

\end{proof}

\begin{lem}\label{lem: q=2}
Let $p$ be an odd prime with $ \max\{v_2(p-1), v_2(p+1)\} = \mu.$ Let $v_2(n)=\nu\geq1.$ Then $v_2(p^n-1)= \mu+\nu.$
\end{lem}

\begin{proof}
We proceed by induction on $v_2(n)$.  For every odd $d$, exactly one of $p^d-1$, $p^d+1$ is divisible by $4$.  (In particular,  $\mu\geq2$.)
Similar to the proof of Lemma~\ref{lem: q notdiv n},     we have
\begin{align*}
v_2(p^{2d}-1) &= v_2(p^d-1) + v_2(p^d+1) \\
& = v_2(p-1) + v_2(\text{odd number}) + v_2(p+1) +  v_2(\text{odd number}) \\
&= \mu+1.
\end{align*}

Assume for all $n$ with $v_2(n)=\nu > 1$ we have $v_2(p^n-1)=\mu+\nu > 1$, in which case $v_2(p^n+1)=1$.
Consider some $n$ with $v_2(n) = \nu + 1$ and choose $d$ odd such that $n=2^{\nu+1}d$.

\begin{align*}
v_2(p^n-1)&=v_2(p^{2^{\nu+1}d}-1) 
=v_2(p^{2^{\nu}d}-1) + v_2(p^{2^{\nu}d}+1)\\
&=\mu+\nu+ 1. \qedhere
\end{align*} 
\end{proof}

\begin{lem}\label{lem: q odd}
Let $q$ be an odd prime.
Suppose $v_q( p^d - 1 ) = \mu \geq 1$ and $v_q(n) = \nu$.  Then $v_q( p^{nd} - 1 ) = \mu + \nu$.
\end{lem}

\begin{proof}
The result for $\nu = 0$ is exactly Lemma~\ref{lem: q notdiv n}.  Choose $k$ so that $p^d = 1+ kq^\mu$ (in particular $q \nmid k$).  Since $q\geq 3$ and $\mu \geq 1$, we have $q \mu \geq \mu +2$.  Hence
\[
p^{qd} = (1 + kq^\mu)^q \equiv 1 + kq^{\mu+1} \pmod{q^{\mu + 2}}.
\]
The result then follows by a straightforward induction.
\end{proof}

%
%

Our main results in Sections~\ref{sec: power} and~\ref{sec: Cheby} will be stated for maps of prime degree~$q$.    The following Lemma shows that in fact the proportion of periodic points is identical for the maps of degree $q$ and degree $q^e$.  We focus on the prime degree case for ease of exposition.

\begin{lem}\label{lem : primepowers}
Let $\phi(z)=z^q$ and $\psi(z)=z^{q^e}$.  Then $\Per(\phi, \FF_{p^n}) = \Per(\psi,\FF_{p^n})$ for every $n$.  Similarly, 
$\Per(T_q, \FF_{p^n}) = \Per(T_{q^e},\FF_{p^n})$
\end{lem}
\begin{proof}
Note that $\phi^m(z) = z^{q^m}$ and $\psi^m(z) = z^{q^{em}}$.  So if $\phi^m(\alpha) = \alpha$, then likewise $\psi^m(\alpha) = \alpha$.  On the other hand, if $\psi^m(\alpha) = \alpha$, then $\phi^{em}(\alpha) = \alpha$.  Applying the iteration for Chebychev polynomials in~\eqref{eqn: cheby iter} gives the result in that case as well.
\end{proof}

\section{Power maps}\label{sec: power}
Throughout this section, we fix the polynomial
\[
\phi(z) = z^{q},
\]
for $q$ prime.  We also take $p$ to be any prime different from $q$.  Our interest is in understanding the proportion of periodic points in $\FF_{p^n}$ as $n$ grows.  In particular, we consider the following limits.

\begin{define}\label{def:periodic proportion}
 We define the following proportions for integers $\nu \geq 0$.  Recall that $\delta$ is the multiplicative order of $p$ modulo $q$.
\[
P_{\nu}(\phi) = \lim_{\substack{
n \to \infty\\
\delta \mid n\\
v_q(n)=\nu }}\frac{\# \Per\left(\phi, \FF_{p^n}\right)}{p^n}.
\]
\end{define}

Since $\delta = \ord_q(p)$, we know that $\delta < q$.  So if $n$ satisfies 
\begin{align*}
\delta \mid n &\text { and }v_q(n) = \nu,\\
\intertext{then there is $n'$ such that }
 n = \delta n' &\text{ and } v_q(n') = \nu.
 \end{align*}
   We will implicitly use this fact later when applying Lemma~\ref{lem: q odd}.

We begin by classifying explicitly the periodic points of $\phi$ in  $\FF_{p^n}$.

\begin{lem}\label{lem: periodic}
Let $p^n -1 = q^ed$ with $q \nmid d$.  Then
\[
\Per(\phi, \FF_{p^n}) = \{ 0 \} \cup \{ \alpha \in \FF_{p^n} \colon \alpha^d = 1 \}.
\]
\end{lem}

\begin{proof}
The defining equation for $\FF_{p^n}$ is 
\begin{equation}\label{eqn : Fpn}
z^{p^n}-z=
z(z^d -1) Q(z),
\end{equation}
for some monic $Q(z) \in \ZZ[z]$.
Clearly $0$ is fixed by $\phi$.  Since $q \nmid d$, the roots of $z^d - 1$ form a group of  order prime to $q$.  Hence $\phi(z) = z^q$ is a permutation of the group elements, and these roots are forward invariant under $\phi$.  So we have 
\[
 \{ 0 \} \cup  \{ \alpha \in \FF_{p^n} \colon \alpha^d = 1 \} \subseteq \Per(\phi, \FF_{p^n}) .
\]
Now let $\alpha$ be a root of $Q(z)$; so in particular $\alpha^{q^ed} = 1$ but $\alpha^d \neq 1$.  Hence for some $1 \leq i \leq e$ and some $d' \mid d$,  we have $\alpha^{q^i d'}= 1$.  In other words, $\alpha^{q^i}$ has order dividing $d$ and  is therefore a root of $z^d -1$.  Since roots of $z^d-1$ are forward invariant under~$\phi$,  $\alpha$ is not periodic for $\phi$.
 \end{proof}
 
 \begin{remark*}
 We applied Lemma~\ref{lem: periodic} to create the examples in Figures~\ref{fig:lots of perper} and~\ref{fig:lots of per}.   Finding a value of $p^n-1$   where, in the notation of the Lemma, $q^e$ is much smaller than $d$ gives ``lots of periodic points.''  Similarly, an example where  $q^e$ is relatively large compared with $d$ gives few periodic points.
 \end{remark*}
 
 The following Proposition justifies our choice of limit in Definition~\ref{def:periodic proportion} because the only interesting proportions of periodic points are those where $\ord_q(p) =\delta \mid n$.
 
 \begin{prop}\label{prop:permpoly}
 If $\ord_q(p) = \delta \nmid n$, all points of $\FF_{p^n}$ are periodic under $\phi$.
 \end{prop}
 
 \begin{proof}
Since $\delta \nmid n$, $q \nmid p^n - 1$.
 The result  follows immediately from Lemma~\ref{lem: periodic}.    
  \end{proof}

We now prove our main results for pure power maps.  The statement is slightly different depending on whether $q=2$ or $q$ is an odd prime.  The difference parallels exactly the difference between the valuation calculations in Lemmas~\ref{lem: q=2} and~\ref{lem: q odd}.

\begin{thm}\label{thm: q=2}
Let  $v_2(p-1)= \lambda$ and $\max\{v_2(p-1), v_2(p+ 1)\} = \mu$. Then for $\phi(z) = z^2$ we have
\begin{align*}
P_0(\phi) &= \frac 1{2^\lambda}, && \text{and}\\
P_{\nu}(\phi)&=\frac{1}{2^{\mu+\nu}} && \text{for }\nu \geq 1.
\end{align*}
\end{thm}

\begin{proof}
First consider $n$ odd.   By Lemma~\ref{lem: q notdiv n}, 
we may choose $d_n$ odd so that $p^n-1=2^\lambda d_n$.   By Lemma ~\ref{lem: periodic} the periodic points for $\phi$ in $\FF_{p^n}$ are $0$ and roots of $z^{d_n}-1$.  
So there are $d_n+1$ points in $\Per(\phi, \FF_{p^n})$. 
Then 
\[
P_0(\phi) = 
\lim_{\substack{
n \to \infty\\
n \text{ odd}}}
\frac{\#\Per(\phi, \FF_{p^n})}{p^n}=
\lim_{d_n \to \infty}
\frac{d_n+1}{2^\lambda d_n+1}
=
\lim_{\substack{
d \to \infty\\
d \text{ odd}}}
\frac{d+1}{2^\lambda d+1}
=\frac{1}{2^\lambda}.  
\]

Now let $v_2(n) = \nu \geq 1$.  By Lemma~\ref{lem: q=2},   $p^n-1=2^{\mu+\nu}d_n$ with $d_n$  odd. Again, the periodic points for $\phi$ in $\FF_{p^n}$ are $0$ and roots of $z^{d_n}-1$.  
Hence
\[
P_{\nu}(\phi)= 
\lim_{\substack{
n \to \infty\\
v_2(n)=\nu}}
\frac{\#\Per(\phi,\FF_{p^n})}{p^n}
=
\lim_{\substack{
d \to \infty\\
d \text{ odd}}}
\frac{d+1}{2^{\mu+\nu}d+1}
=\frac{1}{2^{\mu+\nu}}. \qedhere
\]
\end{proof}

In Tables~\ref{z2tab1}--\ref{z2tab2}, we illustrate Theorem~\ref{thm: q=2}.  The data  were calculated using Sage~\cite{sage}.

\begin{table}[!htp]
\centering
\begin{tabular}{c || c c c c}
$p$ & 3 & 5 & 41 & 17 \\
$\lambda = v_2(p-1)$ & 1 & 2 & 3 & 4\\ \hline \hline
\\[.001in]
$\displaystyle \frac{\#\Per(z^2, \FF_{p})}{p}$ & 0.666666667 \quad &0.400000000 \quad & 0.146341463 
\quad & 0.117647059  \\[.2in]
$\displaystyle \frac{\#\Per(z^2, \FF_{p^3})}{p^3}$ & 0.518518518 & 0.256000000 & 0.125012696 
& 0.0626908203 \\[.2in]
$\displaystyle \frac{\#\Per(z^2, \FF_{p^5})}{p^5}$ & 0.502057613 & 0.250240000 & 0.125000008
&0.0625006603  \\[.2in]
$\displaystyle \frac{\#\Per(z^2, \FF_{p^7})}{p^7}$ & 0.500228624 & 0.250009600 & 0.125000000
& 0.0625000023 \\[.2in] \hline\hline
$\displaystyle \frac{1}{2^\lambda}$ & $0.5$&$ 0.25 $&$ 0.125$ &$ 0.0625$\\[.2in]
\end{tabular}
\caption{$\displaystyle \frac{\#\Per(z^2, \FF_{p^n})}{p^n}$ with $n$ odd.}
\label{z2tab1}
\end{table}

\begin{table}[!htp]
\centering
\begin{tabular}{c || c c c}
$p$ & 3 & 7 & 17 \\
$\mu = \max\{ v_2(p-1), v_2(p+1)\}$ & 2 & 3 & 4\\ \hline \hline
\\[.001in]
$\displaystyle \frac{\#\Per(z^2, \FF_{p^2})}{p^2}$ & 0.222222222 \quad &0.0816326530 \quad & 0.0346020761 
  \\[.2in]
$\displaystyle \frac{\#\Per(z^2, \FF_{p^6})}{p^6}$ & 0.126200274 & 0.0625079686 & 0.0312500401 
 \\[.2in]
$\displaystyle \frac{\#\Per(z^2, \FF_{p^{10}})}{p^{10}}$ & 0.125014818 & 0.0625000033 & 0.0312500000
  \\[.2in]
$\displaystyle \frac{\#\Per(z^2, \FF_{p^{14}})}{p^{14}}$ & 0.125000183 & 0.0625000000 & 0.0312500000
 \\[.2in] \hline\hline
$\displaystyle \frac{1}{2^{\mu+1}}$ & $ 0.125 $&$ 0.0625$ &$ 0.03125$  \\[.2in]
\end{tabular}
\caption{$\displaystyle \frac{\#\Per(z^2, \FF_{p^n})}{p^n}$ with $v_2(n) = 1$.}
\label{z2tab2}
\end{table}

\begin{thm}\label{oddcase}
Let $q$ be an odd prime.
We continue with the earlier notation: $\delta = \ord_q(p)$ and
 $v_q(p^\delta-1)= \mu \geq 1$. For $\phi(z) = z^q$, we have 
\[
P_\nu(\phi) = \frac 1{q^{\mu+\nu}}.
\]
\end{thm}
\begin{proof}
Recall that the limit for $P_\nu(\phi)$ is taken over $n$ such that $\delta \mid n$ and $v_q(n) = \nu$.
By Lemma~\ref{lem: q odd}, for such $n$ we have $p^n - 1 = q^{\mu+\nu}d_n$ with $q \nmid d_n$, and 
by Lemma~\ref{lem: periodic} the periodic points are $0$ and roots of $z^{d_n}-1.$ 
So 
\begin{align*}
P_\nu(\phi) = &
\lim_{\substack{
n \to \infty\\
\delta \mid n\\
v_q(n) = \nu}}
\frac{\#\Per(\phi, \FF_{p^n})}{p^n}=
\lim_{\substack{
n \to \infty\\
\delta \mid n\\
v_q(n) = \nu}}
\frac{d_n+1}{q^{\mu+\nu}d_n+1}\\
&=
\lim_{\substack{
d \to \infty\\
q \nmid d}}
\frac{d+1}{q^{\mu+\nu}d+1}
=\frac{1}{q^{\mu+\nu}}.  \qedhere
\end{align*}
\end{proof}

Tables~\ref{z3tab1}--\ref{z3tab2} illustrate Theorem~\ref{oddcase}  for the map $\phi(z) = z^3$.  Again, the data  were calculated using Sage~\cite{sage}.

\begin{table}[!htp]
\centering
\begin{tabular}{c || c c c }
$p$ &  5 & 19 & 53 \\
$\delta = \ord_3(p)$ & 2 & 1 & 2\\
$\mu = v_3(p^\delta-1)$ & 1 & 2 & 3 \\ \hline \hline
\\[.001in]
$\displaystyle \frac{\#\Per(z^3, \FF_{p^\delta})}{p^\delta}$ & 0.360000000 \quad &0.157894737 \quad & 0.0373798505 
  \\[.2in]
$\displaystyle \frac{\#\Per(z^3, \FF_{p^{2\delta}})}{p^{2\delta}}$ & 0.334400000 & 0.113573407 & 0.0370371591 
 \\[.2in]
$\displaystyle \frac{\#\Per(z^3, \FF_{p^{4\delta}})}{p^{4\delta}}$ & 0.333335040 & 0.111117932 & 0.0370370371
 \\[.2in] \hline\hline
$\displaystyle \frac{1}{3^\mu}$ & $0.333333333$&$ 0.111111111 $&$ 0.0370370370$ \\[.2in]
\end{tabular}
\caption{$\displaystyle \frac{\#\Per(z^3, \FF_{p^n})}{p^n}$ with $v_3(n) = 0$.}
\label{z3tab1}
\end{table}

\begin{table}[!htp]
\centering
\begin{tabular}{c || c c c}
$p$ &  5 & 19 & 53 \\
$\delta = \ord_3(p)$ & 2 & 1 & 2\\
$\mu = v_3(p^\delta-1)$ & 1 & 2 & 3 \\ \hline \hline
\\[.01in]
$\displaystyle \frac{\#\Per(z^3, \FF_{p^{3\delta}})}{p^{3\delta}}$
& 0.111168000 \quad &0.0371774311 \quad & 0.0123456791 
  \\[.2in]
$\displaystyle \frac{\#\Per(z^3, \FF_{p^{6\delta}})}{p^{6\delta}}$ 
& 0.111111115 & 0.0370370575 & 0.0123456790 
 \\[.2in]
$\displaystyle \frac{\#\Per(z^3, \FF_{p^{12\delta}})}{p^{12\delta}}$ 
& 0.111111111 & 0.0370370370 & 0.0123456790
 \\[.2in] \hline\hline
$\displaystyle \frac{1}{3^{\mu+1}}$ & $0.111111111$&$ 0.0370370370 $&$ 0.0123456790$ \\[.2in]
\end{tabular}
\label{z3tab2}
\caption{$\displaystyle \frac{\#\Per(z^3, \FF_{p^n})}{p^n}$ with $v_3(n) = 1$.}
\end{table}

We wish to extend our results to polynomials with composite degree.  Lemma~\ref{lem : primepowers} takes care of prime power degree, so we are left to consider the case $\phi(z)=z^t$ for  $t = q_1^{f_1}q_2^{f_2}\cdots q_r^{f_r}$ and $r \geq 2$.  
For each $1 \leq i \leq r$, let 
\[
\delta_i =\ord_{q_i}(p)\quad  \text{ and}  \quad
\mu_i =v_{q_i}(p^{\delta_i}-1). 
\]
We also define 
\[
\Delta = \lcm\{ \delta_i\}_{1 \leq i \leq r}.
\]
 An argument identical to the one in Proposition~\ref{prop:permpoly2}  shows that if $\gcd(\Delta, n) = 1$, then all points of $\FF_{p^n}$ will be periodic.  Unlike the case of prime degree, however, we need not require $\Delta \mid n$  to have a nontrivial ratio of periodic points.  
 
In order to define the appropriate towers of fields, we need a bit more notation.
 For each nonempty subset $I \subseteq \{1, 2, \ldots, r\}$, let
\[
\delta_I = \lcm\{\delta_i\}_{i \in I} .
\]
If $\delta_I = \delta_{I'}$, then $\delta_{I \cup I'} = \delta_I$ as well.  Hence to  a fixed value of $\delta$ we will associate the \emph{maximal} subset $J\subseteq \{1,2,\ldots, r\}$ such that $\delta_J \mid  \delta$.
 Finally, given an integer $n$, we define an $r$-tuple of valuations
\[
v(n) = \langle v_{q_i}(n)\rangle_{1 \leq i \leq r}.
\]

We now have the tools to define limiting proportions of periodic points  along appropriate towers of finite fields.    Define
\[
P_{\delta,\nu}(\phi) =
 \lim_{\substack{
n \to \infty\\
\gcd(\Delta, n) = \delta\\
v(n)=\langle \nu_i\rangle}}
\frac{\# \Per\left(\phi, \FF_{p^n}\right)}{p^n}.
\]

\begin{prop}\label{compositecase} Let $\phi(z)=z^t$ where $t=q_1^{f_1}q_2^{f_2}\ldots q_r^{f_r}$,   with $q_i$ distinct odd primes for $1\leq i\leq r.$  
 Then for $J\subseteq \{1, 2, \ldots, r\}$ maximal with $\delta_J \mid \delta$,
\[
P_{\delta, \nu}(\phi) = \prod_{j \in J}\frac 1{q_j^{\mu_j+\nu_j}}.
\]
\end{prop}

\begin{remark*}
If no $\delta_i \mid \delta$, then the maximal set $J$ is empty, and we recover the fact that all points in $\FF_{p^n}$ are periodic in this case.  This theorem also recovers our result in Theorem~\ref{oddcase} when applied to the case $t = q$ for $q$ an odd prime.
\end{remark*}

\begin{proof}
Since $J$ is maximal such that $\delta_J \mid \gcd(\Delta, n) $, we have 
\[
p^n-1=d_n \prod_{j \in J} q_j^{e_j} \quad \text{ with } \gcd(t, d_n) = 1.
\]
Lemma~\ref{lem: q odd} shows that $e_j = v_{q_j}(p^n-1) = \mu_j + \nu_j$ for each $j \in J$.

The proof of Lemma~\ref{lem: periodic} extends easily to this case, and we have 
 \[
 \Per\left(\phi,\FF_{p^n}\right)=\{0\}\cup \{\alpha\in \FF_{p^n} : \alpha^{d_n}=1\}.
 \] 
   Hence
\begin{align*}
P_{\delta,\nu}(\phi) &=
 \lim_{\substack{
n \to \infty\\
\gcd(\Delta, n) = \delta\\
v(n)=\langle \nu_i\rangle}}
\frac{\# \Per\left(\phi, \FF_{p^n}\right)}{p^n}.\\
&=
 \lim_{\substack{
d_n \to \infty\\
\gcd(t, d_n) = 1}}
\frac{d_n+1}{d_n \prod_{j \in J} q_j^{\mu_j + \nu_j} + 1 }
=
\prod_{j \in J}\frac 1{q_j^{\mu_j+\nu_j}}.\qedhere
\end{align*}
\end{proof}

In Tables~\ref{z^15tab1}--\ref{z^15tab2}  we use data from Sage~\cite{sage} to illustrate Theorem~\ref{compositecase}  for the map $\phi(z) = z^{15}$ over fields $\FF_{2^n}$.  In the notation of the theorem, we have the following:
\begin{align*}
 q_1 &= 3 & q_2 & = 5 & p &=2 \\
 \delta_1 & = 2 & \delta_2 & = 4 & \Delta & = 4\\
 \mu_1 & = v_3(2^2-1) = 1 & \mu_2&= v_5(2^4-1) = 1.\\
\end{align*}
The table contains values of $n$ with $\gcd(4,n) = \delta$.

\begin{table}[!htp]
\centering
\begin{tabular}{c || c c c }
$\delta $ & 1 & 2 & 4\\ \hline\hline
\\[.001in]
$\displaystyle \frac{\#\Per(z^{15}, \FF_{2^\delta})}{2^\delta}$ & 1.00000000 \quad &0.500000000 \quad & 0.125000000 
  \\[.2in]
$\displaystyle \frac{\#\Per(z^{15}, \FF_{2^{7\delta}})}{2^{7\delta}}$ & 1.00000000 & 0.333374023 & 0.0666666701 
 \\[.2in]
$\displaystyle \frac{\#\Per(z^{15}, \FF_{2^{11\delta}})}{2^{11\delta}}$ & 1.00000000 & 0.333333492 & 0.0666666667
 \\[.2in] \hline\hline
$\{q_j \colon j \in J \}$  & $\emptyset$&  $\{ 3 \}$ &$\{ 3, 5 \}$  
  \\[.2in]
$\displaystyle\prod_{j \in J}\frac 1{q_j^{\mu_j}}$ & $1$&$ 0.333333333 $&$ 0.0666666666$ \\[.2in]
\end{tabular}
\caption{$\displaystyle \frac{\#\Per(z^{15}, \FF_{2^n})}{2^n}$ with $\nu = \left(v_3(n) , v_5(n)\right)=( 0,0)$.}
\label{z^15tab1}
\end{table}

\begin{table}[!htp]
\centering
\begin{tabular}{c || c c c }
$\delta $ & 1 & 2 & 4\\ \hline\hline
\\[.001in]
$\displaystyle \frac{\#\Per(z^{15}, \FF_{2^{3\delta}})}{2^{3\delta}}$ & 1.00000000
 \quad &0.125000000 \quad & 0.0224609375 
  \\[.2in]
$\displaystyle \frac{\#\Per(z^{15}, \FF_{2^{21\delta}})}{2^{21\delta}}$ & 1.00000000
 & 0.111111111 & 0.0222222222 
 \\[.2in]
$\displaystyle \frac{\#\Per(z^{15}, \FF_{2^{33\delta}})}{2^{33\delta}}$ & 1.00000000
 & 0.111111111 & 0.0222222222
 \\[.2in] \hline\hline
$\{q_j \colon j \in J \}$  & $\emptyset$&  $\{ 3 \}$ &$\{ 3, 5 \}$  
  \\[.2in]
$\displaystyle\prod_{j \in J}\frac 1{q_j^{\mu_j + \nu_j}}$ & $1$&$ 0.111111111 $&$ 0.0222222222$ \\[.2in]
\end{tabular}
\caption{$\displaystyle \frac{\#\Per(z^{15}, \FF_{2^n})}{2^n}$ with $\nu = \left(v_3(n) , v_5(n)\right)=(1,0)$.}
\label{z^15tab2}
\end{table}

\begin{remark*}
A statement similar to Proposition~\ref{compositecase} holds when $t$ is even, though the bookkeeping is somewhat messier.  One must apply the results in Lemma~\ref{lem: q=2}, with the exponent for 2 depending on $\max\{v_2(p - 1),v_2( p+1)\}$ and $v_2(n)$.  We leave the details to the interested reader.
\end{remark*}

\section{Chebyshev polynomials}\label{sec: Cheby}

Throughout this section, we consider $T_q(z)$, the Chebyshev polynomial of prime degree $q$.
We take $p$ to be any prime different from $q$.     The proportions of interest in this case run over slightly different towers of finite fields than in the power map case.

\begin{define}\label{def:Cheby proportion}
 We define the following proportions for integers $\nu \geq 0$.  Recall that $\delta$ is the multiplicative order of $p$ modulo $q$.
\[
R_{\nu}(T_q) = \lim_{\substack{
n \to \infty\\
\delta \mid 2n\\
v_q(n)=\nu }}\frac{\# \Per\left(T_q, \FF_{p^n}\right)}{p^n}.\\
\]
\end{define}

We begin with an explicit classification of the  periodic points of $T_q$ in $\overline{\FF_p}$.  For any $\omega \in\overline{\FF_p}$,  we may solve a quadratic to find a nonzero $\zeta \in  \overline{\FF_p}$ such that $\omega = \zeta + \zeta^{-1}$.

\begin{lem}\label{lem: powerequality}
Consider some nonzero $ \zeta \in \overline{\FF_p}$ and an integer $d \geq 0$.  Then 
\begin{equation*}
\zeta + \zeta^{-1} = \zeta^d + \zeta^{-d} \qquad
\text{   if and only if } \qquad
\zeta = \zeta^d \text{ or } \zeta = \zeta^{-d}.
\label{eqn: z+1/z}
\end{equation*}
\end{lem}

\begin{proof}
\begin{align*}
\zeta + \zeta^{-1} &= \zeta^d + \zeta^{-d}\\
\zeta^{2d} - \zeta^{d+1} - \zeta^{d-1} +1 &=0\\
(\zeta^{d-1} - 1) (\zeta^{d+1} - 1) &=0.
\end{align*}
Since $\zeta\neq 0$, the first factor vanishes if and only if $\zeta^d = \zeta$ and the second vanishes if and only if $\zeta^d = 1/\zeta$.
\end{proof}

\begin{lem}\label{lem: T2periodic}
Let $\omega \in \overline{\FF_p}$.  Then $\omega \in  \Per(T_q, \overline{\FF_p})$  if and only if $\omega = \zeta + \zeta^{-1}$ where $\zeta^d = 1$ for some  $d$ relatively prime to $q$.
\end{lem}

\begin{proof} 
 Suppose  $\omega \in \overline{\FF_p}$ is periodic for $T_q$, and choose $\zeta$ so that $\omega = \zeta + \zeta^{-1}$.  Then
\begin{align*}
T_q^n (\omega) &= \omega; \quad \text{that is, }\\
T_{q^n}(\zeta + \zeta^{-1}) &= \zeta^{q^n} + \zeta^{-q^n} = \zeta + \zeta^{-1}.
\end{align*}
So by Lemma~\ref{lem: powerequality}, $\zeta^{q^n-1} = 1$ or $\zeta^{q^n+1} = 1$.

Conversely, suppose there is $d$ prime to $q$ such that $\zeta^d = 1$, and let $\varphi$ be the Euler totient function.  Since  $d \mid \left(q^{\varphi(d)}-1\right)$,
  \[
  \zeta^{q^{\varphi(d)}-1} = 1; \quad \text{ that is, } \quad \zeta^{q^{\varphi(d)}} = \zeta.
  \]
   Hence $\omega = \zeta + \zeta^{-1}$ is fixed by $T_q^{\varphi(d)}$.
\end{proof}

We see that counting the  periodic points for $T_q(z)$ in $\FF_{p^n}$ reduces to counting $\zeta\in \FF_{p^n}$ such that $\zeta + \zeta^{-1} \in \FF_{p^n}$ and $\zeta^d = 1$ for some  $d$ prime to $q$.

\begin{lem}\label{lem: z+1/z in F_p^n}
Let $\zeta \in \overline{\FF_p}$.  Then $\zeta + \zeta^{-1} \in \FF_{p^n}$ if and only if $0 \neq \zeta \in \FF_{p^n} $ or $\zeta^{p^n+1} = 1$.
\end{lem}

\begin{proof}
We have $\zeta + \zeta^{-1}  \in \FF_{p^n}$ if and only if it satisfies
\begin{align*}
\left(\zeta + \zeta^{-1}  \right)^{p^n} &=  \zeta + \zeta^{-1}\\
\zeta^{p^n} + \zeta^{-p^n}  &=  \zeta + \zeta^{-1}.
\end{align*}
So by Lemma~\ref{lem: powerequality} either $\zeta = \zeta^{p^n}$ (i.e. $\zeta \in \FF_{p^n} $) or $1/\zeta = \zeta^{p^n}$.
\end{proof}

Once again, the classification of periodic points explains our choice of limit in Definition~\ref{def:Cheby proportion}.

 \begin{prop}\label{prop:permpoly2}
 If $\ord_q(p) = \delta \nmid 2n$, then all points of $\FF_{p^n}$ are periodic under $T_q$.
 \end{prop}
 
 \begin{proof}
Given that
 \[
q \nmid p^{2n}-1, \quad \text{ we conclude that } \quad q \nmid p^n + 1 \text{ and } q \nmid p^n - 1.
\]   By Lemma~\ref{lem: z+1/z in F_p^n}, every $\omega \in \FF_{p^n}$ can be written as $\zeta + \zeta^{-1}$ for some $\zeta$ with either $\zeta^{p^n-1} = 1$ or $\zeta^{p^n+1} = 1$.  Since $p^n-1$ and $p^n+1$ are both prime to $q$, the result follows from Lemma~\ref{lem: T2periodic}.  
  \end{proof}

We now prove our main results for Chebyshev polynomials.  As in the case of pure power maps, the statements are slightly different in the case $q=2$ versus $q$ odd.  

\begin{thm}\label{lem: P_0 chebyshev}

Let $ \mu = \max\{v_2(p-1), v_2(p+1)\} $. Then
\[
R_\nu(T_2) = \frac{2^{\mu+\nu-1}+1}{2^{\mu+\nu+1}}.
\]
\end{thm}

\begin{proof}
Assume $\omega \in \FF_{p^n}$ is periodic for $T_2$.  Then by Lemma~\ref{lem: T2periodic},   $\omega = \zeta + \zeta^{-1} $, where $\zeta^d = 1$ for some odd $d$.  Since $\zeta + \zeta^{-1} \in \FF_{p^n}$, we apply  Lemma~\ref{lem: z+1/z in F_p^n} to conclude that $\zeta^{p^n+1} = 1$ or $\zeta^{p^n-1} = 1$.

First suppose  $v_2(n)=0$, so by Lemma~\ref{lem: q notdiv n}  $v_2(p-1) = v_2(p^n-1)$.  Then
\begin{equation}
\begin{IEEEeqnarraybox}[][c]{l?s}
\IEEEstrut
p^n-1 = 2^\mu d_1 \\
p^n + 1 =2d_2; 
\IEEEstrut
\end{IEEEeqnarraybox}
\qquad \text{or}\qquad
\begin{IEEEeqnarraybox}[][c]{l?s}
\IEEEstrut
p^n-1 = 2 d_2 \\
p^n + 1 =2^\mu d_1,
\IEEEstrut
\end{IEEEeqnarraybox}
\label{eqn:q=2}
\end{equation}
where $d_1$ and $d_2$ are odd.  
 Note that $d_1$ and $d_2$ are relatively prime since 
$d_1 \mid (p^n+1)$ and $d_2 \mid (p^n-1)$ or vice-versa,
with both odd.

Similarly, if $v_2(n)=\nu \geq 1$,  Lemma~\ref{lem: q=2} shows that $v_2(p^n-1) = \mu + \nu$, so we have
\begin{equation*}
p^n-1 = 2^{\mu+\nu} d_1 \quad \text{ and } \quad 
p^n + 1 =2d_2; 
\end{equation*}
where $d_1$ and $d_2$ are odd and relatively prime.

In either case, $\zeta + \zeta^{-1}$ is periodic if and only if $\zeta^{d_1} = 1$ or $\zeta^{d_2}=1$.  Each such pair  $(\zeta, \zeta^{-1})$ --- including the pair $(1,1)$ ---  corresponds to a periodic point for $T_2$.  Therefore, we have  
$(d_1 + d_2)/2$ periodic points for $T_2$ in $\FF_{p^n}$.

Asymptotically,  $p^n+1 \sim p^n-1$.  That is, 
\[ 
2^{\mu+\nu}  d_1 \sim 2 d_2, \quad \text{ so } \quad 2^{\mu+\nu-1} d_1 \sim d_2.
\]
Hence
\[
R_\nu(T_2) = 
\lim_{
\substack{n \to \infty \\ v_2(n) = \nu}
}
\frac{\# \Per\left(T_2, \FF_{p^n}\right)}{p^n} =
\lim_{
\substack{d \to \infty \\ d\text{ odd}}
}
\frac{(d + 2^{\mu+\nu-1}d)/2}{2^\mu d+1} =
\frac{2^{\mu+\nu-1}+1}{2^{\mu+\nu+1}}. \qedhere
\]
\end{proof}

In Tables~\ref{T2tab1}--\ref{T2tab2}, we illustrate Theorem~\ref{lem: P_0 chebyshev} using data from Sage~\cite{sage}.

\begin{table}[!htp]
\centering
\begin{tabular}{c || c c c}
$p$ & 3 & 7 & 17 \\
$\mu = \max\{ v_2(p-1), v_2(p+1)\}$ & 2 & 3 & 4\\ \hline \hline
\\[.001in]
$\displaystyle \frac{\#\Per(T_2, \FF_{p})}{p}$ & 0.333333333 \quad &0.285714286 \quad & 0.294117647 
  \\[.2in]
$\displaystyle \frac{\#\Per(T_2, \FF_{p^3})}{p^3}$ & 0.370370370 & 0.311953353 & 0.281294525 
 \\[.2in]
$\displaystyle \frac{\#\Per(T_2, \FF_{p^{5}})}{p^{5}}$ & 0.374485597 & 0.312488844 & 0.281250154
  \\[.2in]
$\displaystyle \frac{\#\Per(T_2, \FF_{p^{7}})}{p^{7}}$ & 0.374942844 & 0.312499772 & 0.281250001
 \\[.2in] \hline\hline
$\displaystyle \frac{2^{\mu-1}+1}{2^{\mu+1}}$ & $ 0.375 $&$ 0.3125$ &$ 0.28125$  \\[.2in]
\end{tabular}
\caption{$\displaystyle \frac{\#\Per(T_2, \FF_{p^n})}{p^n}$ with $n$ odd.}
\label{T2tab1}
\end{table}

\begin{table}[!htp]
\centering
\begin{tabular}{c || c c c}
$p$ & 3 & 7 & 17 \\
$\mu = \max\{ v_2(p-1), v_2(p+1)\}$ & 2 & 3 & 4\\ \hline \hline
\\[.001in]
$\displaystyle \frac{\#\Per(T_2, \FF_{p^2})}{p^2}$ & 0.333333333 \quad &0.285714286 \quad & 0.266435986 
  \\[.2in]
$\displaystyle \frac{\#\Per(T_2, \FF_{p^6})}{p^6}$ & 0.312757202 & 0.281251859 & 0.265625010 
 \\[.2in]
$\displaystyle \frac{\#\Per(T_2, \FF_{p^{10}})}{p^{10}}$ & 0.312503175 & 0.281250001 & 0.265625000
  \\[.2in]
$\displaystyle \frac{\#\Per(T_2, \FF_{p^{14}})}{p^{14}}$ & 0.312500039 & 0.281250000 & 0.265625000
 \\[.2in] \hline\hline
$\displaystyle \frac{2^{\mu}+1}{2^{\mu+2}}$& $ 0.3125 $&$ 0.28125$ &$ 0.265625$  \\[.2in]
\end{tabular}
\caption{$\displaystyle \frac{\#\Per(T_2, \FF_{p^n})}{p^n}$ with $v_2(n) = 1$.}
\label{T2tab2}
\end{table}

\begin{thm}\label{lem: P_v chebyshev}
Let $q$ be an odd prime.
Let $ v_q(p^{\delta}-1) = \mu \geq 1$. Then
\[
R_\nu(T_q) = \frac{q^{\mu+\nu}+1}{2 q^{\mu+\nu}}.
\]
\end{thm}

\begin{proof}
Assume $\omega \in \FF_{p^n}$ is periodic for $T_q$.  Then by Lemma~\ref{lem: T2periodic},   $\omega = \zeta + \zeta^{-1} $, where $\zeta^d = 1$ for some  $d$ prime to $q$.  Since $\zeta + \zeta^{-1} \in \FF_{p^n}$, we apply  Lemma~\ref{lem: z+1/z in F_p^n} to conclude that $\zeta^{p^n+1} = 1$ or $\zeta^{p^n-1} = 1$.

Since   $ v_q(p^\delta-1) = \mu \geq 1$ and $v_q(n) = \nu$,  by Lemma~\ref{lem: q odd}  $ v_q(p^{2n}-1) = \mu + \nu \geq 1$.  So 
\[
q \mid p^{2n}-1, \quad \text{ which means that } \quad q \mid p^n-1 \text{ or } q \mid p^n+1 \text{ but not both.}
\]
Therefore
\begin{equation}
\begin{IEEEeqnarraybox}[][c]{l?s}
\IEEEstrut
p^n-1 = q^{\mu+\nu} d_1 \\
p^n + 1 = d_2; 
\IEEEstrut
\end{IEEEeqnarraybox}
\qquad \text{or}\qquad
\begin{IEEEeqnarraybox}[][c]{l?s}
\IEEEstrut
p^n-1 =  d_2 \\
p^n + 1 = q^{\mu+\nu} d_1,
\IEEEstrut
\end{IEEEeqnarraybox}
\label{eqn:qoddprime}
\end{equation}
where $q \nmid d_1 d_2$.

Now, $\zeta + \zeta^{-1}$ is periodic if and only if $\zeta^{d_1} = 1$ or $\zeta^{d_2}=1$.  Each such pair  $(\zeta, \zeta^{-1})$ --- including the pairs $(1,1)$ and $(-1,-1)$ if $p$ odd ---  corresponds to a periodic point for $T_q$.  So we have  
$(d_1 + d_2)/2$ periodic points for $T_q$ in $\FF_{p^n}$.

Again,  $p^n+1 \sim p^n-1$ meaning
\[ 
q^{\mu+\nu}  d_1 \sim  d_2.
\]
Hence
\[
R_\nu(T_q) = 
\lim_{
\substack{n \to \infty \\ \delta \mid 2n \\ v_q(n) = \nu}
}
\frac{\# \Per\left(T_q, \FF_{p^n}\right)}{p^n} =
\lim_{
\substack{d \to \infty \\ q \nmid d}
}
\frac{(d + q^{\mu+\nu}d)/2}{q^{\mu+\nu} d+1} =
\frac{q^{\mu+\nu}+1}{2q^{\mu+\nu}}. \qedhere
\]
\end{proof}

\begin{remark*}
Theorem~\ref{lem: P_0 chebyshev} says that the proportion of periodic points in the appropriate towers for $T_2$ is something slightly more than $1/4$, where the difference  depends on the tower.  Similarly, Theorem~\ref{lem: P_v chebyshev} says that for $q$ an odd prime, the proportion is slightly greater than $1/2$.   We can understand these results a bit more intuitively in the following way.

Consider roots of the polynomials $z^{p^n+1}-1$ and $z^{p^n-1}-1$ over the field $\overline{\FF_p}$.  Equation~\eqref{eqn:qoddprime} shows that for one of the two equations, all roots $\zeta$ yield a periodic point $\zeta+\zeta^{-1}$ for $T_q$.  So we are guaranteed something close to $p^n/2$ periodic points from roots of one of the polynomials, and we pick up a few more from roots of the other polynomial.  A similar explanation for $T_2$ can be derived from equation~\eqref{eqn:q=2}.
\end{remark*}

In Table~\ref{T3tab1}, we illustrate Theorem~\ref{lem: P_v chebyshev} for $T_3(z)$ over various finite fields.  Note that for the choices of primes in the table, $\delta \mid 2n$ for all integers $n$.

\begin{table}[!htp]
\centering
\begin{tabular}{c || c c c}
$p$ &  5 & 19 & 53 \\
$\delta = \ord_3(p)$ & 2 & 1 & 2\\
$\mu = v_3(p^\delta-1)$ & 1 & 2 & 3 \\ \hline \hline
\\[.001in]
$\displaystyle \frac{\#\Per(T_3, \FF_{p})}{p}$ & 0.600000000 \quad &0.578947368 \quad & 0.509433962 
  \\[.2in]
$\displaystyle \frac{\#\Per(T_3, \FF_{p^2})}{p^2}$ & 0.680000000 & 0.556786704 & 0.518689925 
 \\[.2in]
$\displaystyle \frac{\#\Per(T_3, \FF_{p^{4}})}{p^{4}}$ & 0.667200000 & 0.555558966 & 0.518518579
 \\[.2in] \hline\hline
$\displaystyle \frac{3^{\mu}+1}{2\cdot 3^{\mu}}$ & $ 0.666666667 $&$ 0.555555556$ &$ 0.518518519$  \\[.2in]
\end{tabular}
\caption{$\displaystyle \frac{\#\Per(T_3, \FF_{p^n})}{p^n}$ with $v_3(n) = 0$.}
\label{T3tab1}
\end{table}

Once again,
we wish to extend our results to polynomials with composite degree.  Lemma~\ref{lem : primepowers} takes care of prime power degree, so we are left to consider the case of the $t^{\text{th}}$ Chebyshev polynomial, $T_t(z)$, for  $t = q_1^{f_1}q_2^{f_2}\cdots q_r^{f_r}$ and $r \geq 2$.  We continue with the notation introduced at the end of Section~\ref{sec: power}:
for each $1 \leq i \leq r$, let 
\[
\delta_i =\ord_{q_i}(p)\quad  \text{ and}  \quad
\mu_i =v_{q_i}(p^{\delta_i}-1). 
\]
We also define 
\[
\Delta = \lcm\{ \delta_i\}_{1\leq i \leq r}.
\]
 The argument in Proposition~\ref{prop:permpoly2} can be modified to  show that if $\gcd(\Delta, 2n) = 1$, then all points of $\FF_{p^n}$ will be periodic.  But as in Section~\ref{sec: power}, we need not require $\Delta \mid 2n$  to have a nontrivial ratio of periodic points.  
 
 As before,  for each $n \in \ZZ$ we define an $r$-tuple of valuations
\[
v(n) = \langle v_{q_i}(n)\rangle_{1 \leq i \leq r}.
\]
We then define the ratios of interest:
\[
R_{\delta,\nu}(T_t) =
 \lim_{\substack{
n \to \infty\\
\gcd(\Delta, n) = \delta\\
v(n)=\langle \nu_i\rangle}}
\frac{\# \Per\left(T_t, \FF_{p^n}\right)}{p^n}.
\]

\begin{thm}\label{lem: R_v chebyshevcomposite}
 Let $t=q_1^{f_1}q_2^{f_2}\ldots q_r^{f_r}$,   with $q_i$ distinct odd primes for $1\leq i\leq r.$   Then there are disjoint subsets $I, J \subseteq \{1, 2, \ldots, r \}$  such that 
 \[
R_{\delta,\nu}(T_t) = \frac{Q_{I} + Q_{J}}{2Q_IQ_J},
\]
where
\[
Q_I = \prod_{i\in I}{q_i^{\mu_i + \nu_i}}
\quad  \text{ and } \quad
Q_{J} = \prod_{j\in J}{q_j^{\mu_j+ \nu_j}}. 
\]
  \end{thm}

\begin{proof}
Take $J$  maximal with $\delta_J \mid \delta$; then we know that $q_j \mid p^\delta - 1$ if and only if $j \in J$.  Now define 
\[
I =\left\{ 1 \leq i \leq r \colon q_i \mid p^\delta + 1 \right\}.
\]
Since the primes dividing $t$ are distinct odd primes,  no $q_i$ divides both $p^\delta - 1$ and $p^\delta+1$.  Hence $I \cap J = \emptyset$.

Now consider any $n$ with $\gcd(\Delta, n) = \delta$.  Clearly $q_j \mid p^n - 1$ if and only if $j \in J$.  For any $i \in I$, we have 
\[
q_i \mid p^\delta + 1 \quad \Longrightarrow \quad q_i \mid p^{2\delta}-1 \quad \Longrightarrow \quad q_i \mid p^{2n} - 1.
\]
Since $i \not\in J$,  $q_i \nmid p^n - 1$.  Therefore $q_i \mid p^n + 1$.  Furthermore, since $\gcd(\Delta, 2n) \mid 2\delta$, we have 
\[
q_i \mid p^{2n}-1 \quad \Longleftrightarrow \quad q_i \mid p^{2\delta}-1 \quad \Longleftrightarrow \quad i \in I \cup J.
\]
That is,
$q_i \mid p^n + 1 $ if and only if $i \in I$.

Therefore 
\begin{align*}
p^n - 1 &= d_1 \prod_{j\in J} q_j^{e_j}  & p^n + 1 &= d_2 \prod_{i \in I} q_i^{e_i},
\end{align*}
with $\gcd(t, d_1) = \gcd(t, d_2) = 1$.  Lemma~\ref{lem: q odd}, applied to $n$ and $2n$ respectively, shows that $e_j = \mu_j+\nu_j$ for $j \in J$ and 
$e_i = \mu_i+\nu_i$ for $i \in I$.

Lemma~\ref{lem: T2periodic} extends easily to the case of composite degree, and we conclude that $\omega \in \FF_{p^n}$ is periodic for $T_t$ if and only if $\omega = \zeta + \zeta^{-1}$ with $\zeta^{d_1} = 1$ or $\zeta^{d_2}=1$.  As before, we have $(d_1+d_2)/2$ periodic points for $T_t$ in $\FF_{p^n}$.

Since $p^n + 1 \sim p^n - 1$, we have
\[
d_1 \prod_{j\in J} q_j^{\mu_j + \nu_j} \sim d_2 \prod_{i \in I} q_i^{\mu_i+ \nu_i},
\quad \text{ meaning } \quad
d_2 \sim d_1 \frac{Q_J}{Q_I}.
\]
We can now calculate the limit:
\begin{align*}
R_{\delta,\nu}(T_t) 
&=
 \lim_{\substack{
n \to \infty\\
\gcd(\Delta, n) = \delta\\
v(n)=\langle \nu_i\rangle}}
\frac{\# \Per\left(T_t, \FF_{p^n}\right)}{p^n}\\
&=
\lim_{
\substack{d \to \infty \\ \gcd(t, d) = 1}
}
\frac{\left(d + d\frac{Q_J}{Q_I}\right)/2}{Q_Jd+1} =
\frac{Q_I + Q_J}{2Q_IQ_J}. \qedhere
\end{align*}

\end{proof}

In Table~\ref{T15tab1}  we use data from Sage~\cite{sage} to illustrate Theorem~\ref{lem: R_v chebyshevcomposite}  for the $15^{\text{th}}$ Chebyshev polynomial over fields $\FF_{2^n}$.  In the notation of the theorem, we have:
\begin{align*}
 q_1 &= 3 & q_2 & = 5 & p &=2 \\
 \delta_1 & = 2 & \delta_2 & = 4 & \Delta & = 4\\
 \mu_1 & = v_3(2^2-1) = 1 & \mu_2&= v_5(2^4-1) = 1.\\
\end{align*}
Note that in the table, we restrict to values of $n$ with $\gcd(4,n) = \delta$.

\begin{table}[!htp]
\centering
\begin{tabular}{c || c c c }
$\delta $ & 1 & 2 & 4\\ 
$2^\delta - 1$ & 1 & 3 & $3\cdot5$ \\
$2^\delta + 1$ & 3 & 5& 17 \\ \hline\hline
\\[.001in]
$\displaystyle \frac{\#\Per(T_{15}, \FF_{2^\delta})}{2^\delta}$ & 0.500000000 \quad &0.266662598 \quad & 0.562500000 
  \\[.2in]
$\displaystyle \frac{\#\Per(T_{15}, \FF_{2^{7\delta}})}{2^{7\delta}}$ & 0.656250000 & 0.266666651 & 
0.506667137 
 \\[.2in]
$\displaystyle \frac{\#\Per(T_{15}, \FF_{2^{11\delta}})}{2^{11\delta}}$ & 0.664062500 & 0.266666667 & 0.533333335
 \\[.2in] \hline\hline
 $\{q_i \colon i \in I \}$ & $\{ 3 \}$ & $\{ 5 \}$ &$ \emptyset$
  \\[.1in]
$\{q_j \colon j \in J \}$  & $\emptyset$&  $\{ 3 \}$ &$\{ 3, 5 \}$ 
 \\[.1in]
$\displaystyle \frac{Q_I + Q_J}{2Q_IQ_J}$ & $0.666666667$&$ 0.266666667 $&$ 0.533333333$ \\[.2in]
\end{tabular}
\caption{$\displaystyle \frac{\#\Per(T_{15}, \FF_{2^n})}{2^n}$ with $\nu = \left(v_3(n) , v_5(n)\right)=( 0,0)$.}
\label{T15tab1}
\end{table}

\newpage

\bibliographystyle{plain}

\begin{thebibliography}{1}

\bibitem{flynngarton}
Ryan Flynn and Derek Garton.
\newblock Graph components and dynamics over finite fields.
\newblock arXiv:1108.4132 [math.NT], August 2011.

\bibitem{madhu}
Kalyani Madhu.
\newblock {\em Galois Theory and Polynomial Orbits}.
\newblock PhD thesis, University of Rochester, 2011.

\bibitem{ads}
Joseph~H. Silverman.
\newblock {\em The arithmetic of dynamical systems}, volume 241 of {\em
  Graduate Texts in Mathematics}.
\newblock Springer-Verlag, 2007.
\newblock To appear.

\bibitem{sage}
W.\thinspace{}A. Stein et~al.
\newblock {\em {S}age {M}athematics {S}oftware ({V}ersion 4.7.2)}.
\newblock The Sage Development Team, 2011.
\newblock {\tt http://www.sagemath.org}.

\end{thebibliography}

\end{document}